\documentclass[12pt]{amsart}
\usepackage{amssymb}

\newtheorem{theorem}{Theorem}[section]
\newtheorem{definition}[theorem]{Definition}
\newtheorem{proposition}[theorem]{Proposition}
\newtheorem{conjecture}[theorem]{Conjecture}
\newtheorem{heuristic}[theorem]{Heuristic}

\begin{document}

\title[Quasi-flat representations]{Quasi-flat representations of uniform groups and quantum groups}

\author{Teodor Banica}
\address{T.B.: Department of Mathematics, University of Cergy-Pontoise, F-95000 Cergy-Pontoise, France. {\tt teo.banica@gmail.com}}

\author{Alexandru Chirvasitu}
\address{A.C.: Department of Mathematics, 244 Mathematics Building, University at Buffalo, Buffalo, NY 14260-2900, USA. {\tt achirvas@buffalo.edu}}

\begin{abstract}
Given a discrete group $\Gamma=<g_1,\ldots,g_M>$ and a number $K\in\mathbb N$, a unitary representation $\rho:\Gamma\to U_K$ is called quasi-flat when the eigenvalues of each $\rho(g_i)\in U_K$ are uniformly distributed among the $K$-th roots of unity. The quasi-flat representations of $\Gamma$ form altogether a parametric matrix model $\pi:\Gamma\to C(X,U_K)$. 

We compute here the universal model space $X$ for various classes of discrete groups, notably with results in the case where $\Gamma$ is metabelian. We are particularly interested in the case where $X$ is a union of compact homogeneous spaces, and where the induced representation $\tilde{\pi}:C^*(\Gamma)\to C(X,U_K)$ is stationary in the sense that it commutes with the Haar functionals. We present several positive and negative results on this subject.

We also discuss similar questions for the discrete quantum groups, proving a stationarity result for the discrete dual of the twisted orthogonal group $O_2^{-1}$.
\end{abstract}


\maketitle

\section*{Introduction}

Interesting quantum analogues of the compact Lie groups $G\subset U_N$ were introduced by Woronowicz in \cite{wo1,wo2}. While the Lie theory is lacking in general, for such quantum groups, various representation-theoretic tools such as Peter-Weyl theory and Tannakian duality are available. These can be deployed in the study of the algebraic relations between the standard coordinates $u_{ij}\in C(G)$, partially recovering some of the tractability of function algebras on ordinary compact Lie groups. 

This avenue of investigation naturally suggests connections to probability theory, both classical and noncommutative (or free). Indeed, the knowledge of the Schur-Weyl dual of $G$ allows one to explicitly compute the integrals of type $\int_Gu_{i_1j_1}\ldots u_{i_kj_k}$ via an extension of the Weingarten formula \cite{csn,wei}, and this leads to a number of probabilistic applications. The geometric meaning of these computations remains, however, largely mysterious.

One idea going back to \cite{ba1} is that one could probe the nature of the algebra $C(G)$ by finding matrix models for the coordinates $u_{ij}\in C(G)$. A priori this is a notion which is interesting only when $C(G)$ is of type I, but due to the notion of ``inner faithfulness'' appearing in \cite{ba1} and later axiomatized and studied in \cite{bb1}, there are in fact no restrictions on $C(G)$. More precisely, consider an arbitrary matrix model for $C(G)$, in the sense that we have a representation of algebras as follows, with $X$ being a compact space:
\begin{equation*}
\pi:C(G)\to M_K(C(X))  
\end{equation*}

If we consider the dual discrete quantum group $\Gamma=\widehat{G}$, the inner faithfulness property states that the model must be faithful on $\Gamma\subset C^*(\Gamma)=C(G)$ instead of being faithful globally. Examples of such models abound, notably for certain non-type-I (and in fact non-amenable) algebras $C(G)$. In fact, there is no known restriction on the class of algebras of type $C(G)$ which can be modelled in this manner.

The interest in the topic stems from the fact that even in the presence of this weak notion of faithfulness, matrix models recover a lot of interesting information about $G$. In fact $G$ itself can be recovered from the model through abstract Tannakian formalism, and we direct here the reader to \cite{bb1,bb2,bfs} for a number of standard results, allowing one to read off various properties of $G$ from its matrix model.

The natural candidates for such a study are the quantum permutation groups $G\subset S_N^+$. There are a number of mathematical and physical motivations for this choice, going back to \cite{ba1} and subsequent papers. For such a quantum group the simplest possible models are those which are ``quasi-flat'' in the sense that the standard coordinates $u_{ij}\in C(G)$, known to be projections, are mapped into projections of rank $\leq1$. Some work on the quasi-flat models was recently done in \cite{ba2,bb2,bch,bfr,bne}. 

Matrix models provide another possible connection to the literature via the study of {\it character} or {\it representation varieties} for discrete groups \cite{km1,km2,nak1,nak2,rap,sik}. These are spaces (typically algebraic varieties) parametrizing the linear representations of $\Gamma$ in much the same way that our space $X$ above does. For this reason, the present work can be regarded as an investigation of varieties of sufficiently well-behaved (i.e. quasi-flat) representations of discrete {\it quantum} groups. 

We will review here some of the basic results and conjectures on such models with a number of new contributions, either novel or as enhancements of prior work. Our main results will concern the notion of stationarity, which asks that the matrix model be compatible with the canonical integration functionals, via a formula as follows, for an appropriately chosen integration measure on $X$:
$$\int_G=\left(tr\otimes\int_X\right)\pi$$

This is a rather strong condition, implying for instance that $\pi$ is faithful. In fact, as explained in \cite{bch}, for most known examples this condition entails the von Neumann algebra $L^\infty(G)$ being of type I, with the group dual case $G=\widehat{\Gamma}$ corresponding to Thoma's theorem \cite{tho}. A fully understanding of this notion is therefore a matter of general interest, for instance in connection with von Neumann's type I-II-III philosophy \cite{mvo,von,kap}.

The paper is organized as follows: Sections 1-2 mostly contain preliminary material, in Sections 3-4-5 we prove that the universal matrix models of the dihedral groups are stationary, and then we prove a number of positive and negative results regarding more general classes of metabelian groups, and finally in Sections 6-7 we show that the discrete quantum dual of the twisted orthogonal group $O_2^{-1}$ admits a stationary matrix model, basing our analysis on the techniques from our previous paper \cite{bch}. 

\medskip

\noindent {\bf Acknowledgements.} A.C. is grateful for partial support from the NSF through grant DMS-1565226.

\section{Matrix models}

We use Woronowicz's quantum group formalism in \cite{wo1}, \cite{wo2}, with the extra assumption $S^2=id$. An extra source of useful information comes from the more recent papers of Maes and Van Daele \cite{mvd} and Malacarne \cite{mal}, which review this material, with a few simplifications. There is as well the book by Neshveyev and Tuset \cite{ntu}.

We recall that a magic unitary matrix is a square matrix over a $C^*$-algebra, $u\in M_N(A)$,  whose entries are projections ($p^2=p^*=p$), summing up to $1$ on each row and each column. The following key definition is due to Wang \cite{wa1}:

\begin{definition}\label{def.sn}
$C(S_N^+)$ is the universal $C^*$-algebra generated by the entries of an $N\times N$ magic unitary matrix $u=(u_{ij})$, with the morphisms given by
$$\Delta(u_{ij})=\sum_ku_{ik}\otimes u_{kj}\quad,\quad\varepsilon(u_{ij})=\delta_{ij}\quad,\quad S(u_{ij})=u_{ji}$$
as comultiplication, counit and antipode. 
\end{definition}

This algebra satisfies Woronowicz' axioms, and the underlying compact quantum group $S_N^+$ is called quantum permutation group. We have an inclusion $S_N\subset S_N^+$, given at the algebra level by $u_{ij}=\chi(\sigma\in S_N|\sigma(j)=i)$, which is an isomorphism at $N=2,3$, but not at $N\geq4$, where $S_N^+$ is non-classical, and infinite. See \cite{bb2}, \cite{bfs}, \cite{wa2}.

Any closed subgroup $G\subset S_N^+$ can be thought of as ``acting'' on the set $\{1,\ldots,N\}$, and one can talk about the orbits of this action. The theory here was developed in \cite{bic}, and also recently in \cite{bfr}. In what follows, we will only need the following notions:

\begin{definition}
Let $G\subset S_N^+$ be a closed subgroup, with magic unitary $u=(u_{ij})$, and consider the equivalence relation on $\{1,\ldots,N\}$ given by $i\sim j\iff u_{ij}\neq0$.
\begin{enumerate}
\item The equivalence classes under $\sim$ are called orbits of $G$.

\item $G$ is called transitive when the action has a single orbit. 

\item $G$ is called quasi-transitive when all the orbits have the same size.
\end{enumerate}
\end{definition}

Here the fact that $\sim$ as defined above is indeed an equivalence relation follows by applying $\Delta,\varepsilon,S$ to a formula of type $u_{ij}\neq0$. For details, see \cite{bfr}.

In the classical case, $G\subset S_N$, we recover in this way the usual notions of orbits, transitivity, and quasi-transitivity. In general, there are many interesting examples of closed subgroups $G\subset S_N^+$ which are transitive, or at least quasi-transitive.

At the level of the general theory, we have the following result, from \cite{bfr}:

\begin{proposition}
The following quantum groups are quasi-transitive: 
\begin{enumerate}
\item Those of product type: the intermediate quantum groups $G_1\times\ldots\times G_M\subset G\subset G_1\,\hat{*}\,\ldots\,\hat{*}\,G_M$, with $G_1,\ldots,G_M\subset S_K^+$ assumed to be transitive.

\item Those of induced type: the normal closed quantum subgroups $H\triangleleft G$, with $G\subset S_N^+$ assumed to be transitive.
\end{enumerate}
\end{proposition}

\begin{proof}
These results are both elementary, the idea being as follows:

(1) This is trivial, because by \cite{wa1} both the usual product $G_1\times\ldots\times G_M$ and the dual free product $G_1\,\hat{*}\,\ldots\,\hat{*}\,G_M$ are quasi-transitive, with orbits of size $K$, and since $G$ sits in between, the coordinates $u_{ij}$ vanish exactly outside the corresponding $K\times K$ blocks.

(2) This is something elementary in the classical case, where the orbits of a group $H\triangleleft G\curvearrowright\{1,\ldots,N\}$, with $G$ transitive, can be put in explicit bijection. In general, this can be proved by using the theory of normal quantum subgroups in \cite{cdp}. See \cite{bfr}.
\end{proof}

We should mention that, unlike in the classical case, the above notions are quite tricky, and there are several open problems regarding them. Of interest for instance is the study of these notions for the quantum automorphism groups of the finite graphs, and there are many unsolved questions here, waiting to be studied. See \cite{bfr}, \cite{bic}, \cite{cha}.

Given a closed subgroup $G\subset S_N^+$, we will be interested here in the matrix models for the algebra $C(G)$. There are several known constructions of such models, and in the quasi-transitive case, the ``simplest'' models are as follows:

\begin{definition}
Let $G\subset S_N^+$ be quasi-transitive, with orbits having size $K$.
\begin{enumerate}
\item A matrix model $\pi:C(G)\to M_K(C(X))$, with $X$ being compact, is called quasi-flat when $P_{ij}=\pi(u_{ij})$ are such that each fiber $P_{ij}^x\in M_K(\mathbb C)$ is of rank $\leq1$.

\item The universal quasi-flat matrix model for $C(G)$, obtained by using the Tannakian relations which define $G$, is denoted $\pi:C(G)\to M_K(C(X_G))$.
\end{enumerate}
\end{definition}

In order to comment on these notions, assume first that $G\subset S_N^+$ is transitive. Given a model $\pi:C(G)\to M_N(C(X))$, mapping $u_{ij}\to P_{ij}^x$, the matrices $(d^x)_{ij}=tr(P_{ij}^x)$ are all bistochastic, with sum $1$. The simplest situation is that when $d^x=(1/N)_{ij}$ is the flat matrix, for any $x\in X$, and in this case we call our model ``flat''.

In the non-transitive case we cannot have flat models, simply because $u_{ij}=0$ implies $P_{ij}^x=0$, for any $x\in X$. However, assuming that $G\subset S_N^+$ is quasi-transitive, with orbits of size $K$, we can consider models of type   $\pi:C(G)\to M_K(C(X))$, with the assumption $d^x_{ij}\leq1/K$ for any $i,j,x$. Thus, we are led to the quasi-flatness notion in (1).

Regarding now (2), here the fact that the universal quasi-flat model exists, is unique, and appears as in the statement is a straightforward consequence of Woronowicz's Tannakian duality results in \cite{wo2}. We refer to \cite{bfr}, \cite{bne}, \cite{mal} for details here.

We would like to understand the faithfulness properties of the various quasi-flat models, including those of the universal one. We use the following notions:

\begin{definition}
A matrix model $\pi:C(G)\to M_K(C(X))$ is called:
\begin{enumerate}
\item Inner faithful, when there is no factorization $\pi:C(G)\to C(H)\to M_K(C(X))$, with $H\subset G$ being a proper closed subgroup.

\item Stationary, when the Haar integration over $G$ appears as $\int_G=(tr\otimes\int_X)\pi$, where $\int_X$ is the integration with respect to a probability measure on $X$.
\end{enumerate}
\end{definition}

These notions are both quite subtle. Regarding (1), in the group dual case, $G=\widehat{\Gamma}$, our model must come from a group representation $\rho:\Gamma\to C(X,U_K)$, and the inner faithfulness of $\pi$ means precisely that $\rho$ must be faithful. In general, what we have here is an extension of this fact. As for (2), the notion there, and the terminology, come from the idempotent state work on the inner faithfulness property in \cite{ba2}, \cite{bfs}, \cite{wa3}, to be explained later on. Let us just mention here, as a basic fact regarding the stationarity, that this property implies the faithfulness. See \cite{ba2}, \cite{bb1}, \cite{bch}, \cite{bfs}, \cite{bfr}, \cite{chi}, \cite{wa3}.

As an illustration, let us first discuss the classical case. With the convention that we identify the rank one projections in $M_K(\mathbb C)$ with the corresponding elements of the complex projective space $P^{K-1}_\mathbb C$, we have the following result, from \cite{bfr}:

\begin{proposition}
Given a quasi-transitive group $G\subset S_N$, with orbits having size $K$, the associated universal quasi-flat model space is $X_G=E_K\times L_{N,K}^G$, where:
$$E_K=\left\{P_1,\ldots,P_K\in P^{K-1}_\mathbb C\Big|P_i\perp P_j,\forall i,j\right\}$$
$$L_{N,K}^G=\left\{\sigma_1,\ldots,\sigma_K\in G\Big|\sigma_1(i),\ldots,\sigma_K(i)\ {\rm distinct},\forall i\in\{1,\ldots,N\}\right\}$$
In addition, assuming that we have $L_{N,K}^G\neq\emptyset$, the universal quasi-flat model is stationary, with respect to the Haar measure on $E_K$ times the discrete measure on $L_{N,K}^G$.
\end{proposition}

\begin{proof}
The key remark here is that two commuting rank 1 projections must be either equal, or orthogonal. Thus, a quasi-flat model for $C(G)$ must be of the form $u_{ij}\to P_{L_{ij}}$, with $P\in E_K$ and with $L\in M_N(*,1,\ldots,K)$ being a ``sparse Latin square'', with the convention $P_*=0$, and this gives the result. See \cite{bfr}.
\end{proof}

We recall now from Bichon's paper \cite{bic} that the group dual subgroups $\widehat{\Gamma}\subset S_N^+$ appear from the quotients of type $\mathbb Z_{K_1}*\ldots*\mathbb Z_{K_M}\to\Gamma$, with $N=K_1+\ldots+K_M$, via a Fourier transform type construction. This result can be used in order to characterize the group duals $\widehat{\Gamma}\subset S_N^+$ which are quasi-transitive, and then to investigate the quasi-flat models for the corresponding algebras $C(\widehat{\Gamma})=C^*(\Gamma)$. The result, from \cite{bch}, is as follows:

\begin{theorem}
The quasi-transitive group duals $\widehat{\Gamma}\subset S_N^+$, with orbits having $K$ elements, have the following properties:
\begin{enumerate}
\item They come from the quotients $\mathbb Z_K^{*M}\to\Gamma$, with $M=N/K$, having the property that the corresponding $M$ morphisms $\mathbb Z_K^{(i)}\subset\mathbb Z_K^{*M}\to\Gamma$ are all injective.

\item For such a quotient, a matrix model $\pi:C^*(\Gamma)\to M_K(\mathbb C)$ is quasi-flat if and only if it is stationary on each subalgebra $C^*(\mathbb Z_K^{(i)})\subset C^*(\Gamma)$.

\item Equivalently, when writing $\mathbb Z_K^{(i)}=<g_i>$, each of the matrices $\pi(g_i)\in U_K$ must has its eigenvalues uniformly distributed over the $K$-th roots of unity.

\item More generally, $\pi:C^*(\Gamma)\to M_K(C(X))$ is quasi-flat when the associated unitary representations $\rho_x:\Gamma\to U_K$ all satisfy the ``quasi-flatness'' condition in (3).
\end{enumerate}
\end{theorem}

\begin{proof}
Here (1) follows from the above-mentioned result from \cite{bic}, (2) follows via an elementary Fourier transform computation, (3) follows by interpreting the stationarity condition found in (2), and finally (4) follows from (3). For details here, see \cite{bch}.
\end{proof}

The above result provides us with a whole new point of view on the quasi-flat models. Indeed, let us first axiomatize the condition found in (3) above:

\begin{definition}
Given a finitely generated discrete group $\Gamma=<g_1,\ldots,g_M>$, we call a parametric unitary representation $\rho:\Gamma\to C(X,U_K)$ quasi-flat when the eigenvalues of each $\rho_x(g_i)\in U_K$ are uniformly distributed among the $K$-th roots of unity.
\end{definition}

Observe that, assuming that $\rho$ as above is faithful, the generators $g_1,\ldots,g_M$ must satisfy $g_i^K=1$ for any $i$. Thus, while this definition is formulated for any $\Gamma$, its range of applications is limited to the case where we have a quasi-flat embedding $\widehat{\Gamma}\subset S_{KM}^+$.

With this picture in hand, which is purely group-theoretical, our general quasi-flat models, as axiomatized in Definition 1.4 above, simply appear via a ``quantum extension of this notion'', by replacing $\Gamma$ with an arbitrary discrete quantum group.

We are of course mostly interested in understanding when these models satisfy the various notions of faithfulness from Definition 1.5. The subject here is non-trivial, and the above results, together with some other results from \cite{bch}, \cite{bfr}, \cite{bne}, which are more technical and will be explained later on, suggest the following conjecture:

\begin{conjecture}
Assume that $G\subset S_N^+$ is quasi-transitive, with orbits having size $K$, and consider the universal flat model $\pi:C(G)\to M_K(C(X_G))$.
\begin{enumerate}
\item If $G$ satisfies suitable ``transitivity type'' assumptions, $\pi$ is inner faithful.

\item If $\widehat{G}$ satisfies suitable ``virtual abelianity'' assumptions, $\pi$ is stationary.
\end{enumerate}
\end{conjecture}

Regarding the evidence, in the classical case, $G\subset S_N$, everything about $\pi$ is of course known from Proposition 1.6 above. The question left is that of understanding the precise meaning of the ``transitivity type'' condition found there, as well as its interpretation as an ``virtual abelianity'' condition regarding the discrete dual $\widehat{G}$. See \cite{bch}, \cite{bfr}.

In the group dual case, Theorem 1.7 above reformulates everything in terms of usual group representations, and the computations in \cite{bfr} provide some evidence for (1). As for (2), as explained in \cite{bch}, this is related to Thoma's theorem \cite{tho}, which states that a group algebra $C^*(\Gamma)$ is of type I precisely when $\Gamma$ is virtually abelian.

Finally, as explained in \cite{bne}, for $G=S_N^+$ itself the question (1) is a quite difficult one. Indeed, having an inner faithful model for $C(S_N^+)$ would imply that the algebra $L^\infty(S_N^+)$ has the Connes embedding property, therefore solving an old open problem. For more details on this question, and for some strategies for dealing with it, see \cite{bne}, \cite{bcv}.

\section{Capturing results}

A faithful model $\Gamma\subset U_K$, or more generally a faithful model $\Gamma\subset C(X,U_K)$, captures everything about a discrete group $\Gamma$, simply because it captures $\Gamma$ itself. In the discrete quantum group case the same holds, because when $\pi:C^*(\Gamma)\to M_K(C(X))$ is inner faithful, the compact dual $G=\widehat{\Gamma}$ has a simple Tannakian description. See \cite{bb1}.

A more concrete point of view on these questions comes from analysis, by assuming that the model space has a probability measure. We have indeed the following result, which reminds Woronowicz's construction of the Haar state in \cite{wo1}, via a Ces\`aro limit:

\begin{proposition}
A matrix model $\pi:C(G)\to M_K(C(X))$, with $X$ being assumed to be a compact probability space, is inner faithful if and only if
$$\int_G=\lim_{k\to\infty}\frac{1}{k}\sum_{r=1}^k\int_G^r$$
where $\int_G^r=(\varphi\circ\pi)^{*r}$, with $\varphi=tr\otimes\int_X$ being the random matrix trace.
\end{proposition}

\begin{proof}
This was proved in \cite{bfs} in the case $X=\{.\}$, by using basic theory from \cite{fsk}, and the general case was recently discussed in \cite{wa3}, by using more advanced tools, the idea being that the Ces\`aro limit in the statement is the Haar functional of the Hopf image.
\end{proof}

In discrete quantum group terms, any property of $\Gamma$ which can be recovered from the explicit knowledge of the Haar functional $\int_{\widehat{\Gamma}}:C^*(\Gamma)\to\mathbb C$ can be ``recaptured'' via the above result from the knowledge of an inner faithful model for $C^*(\Gamma)$. 

In order to formulate some concrete results, we will need:

\begin{proposition}
Assuming that $\pi:C(G)\to M_K(C(X))$ is inner faithful, mapping $u_{ij}\to U_{ij}$, the above truncated integration functionals $\int_G^r$ are given by
$$\int_G^ru_{i_1j_1}\ldots u_{i_pj_p}=(T_p^r)_{i_1\ldots i_p,j_1\ldots j_p}$$
where $T_p\in M_{N^p}(\mathbb C)$ is given by $(T_p)_{i_1\ldots i_p,j_1\ldots j_p}=\left(tr\otimes\int_X\right)(U_{i_1j_1}\ldots U_{i_pj_p})$.
\end{proposition}

\begin{proof}
This follows from an elementary computation, by using the definition of the truncated integrals, namely $\int_G^r=(\varphi\circ\pi)^{*r}$, with $\varphi=tr\otimes\int_X$. See \cite{bb2}, \cite{bfs}.
\end{proof}

In the quasi-flat case, that we are interested in, we can write $U_{ij}^x= Proj(\xi_{ij}^x)$, for certain vectors $\xi_{ij}^x\in\mathbb C^K$, satisfying $||\xi_{ij}^x||\in\{0,1\}$ for any $i,j,x$. We obtain:

\begin{proposition}
Assuming that $\pi:C(G)\to M_K(C(X))$ is inner faithful and quasi-flat, mapping $u_{ij}\to Proj(\xi_{ij}^x)$, with $||\xi_{ij}^x||\in\{0,1\}$, the above matrices $T_p$ are given by
$$T_p=\int_XT_p(\xi^x)dx$$
where the matrix $T_p(\xi)\in M_{N^p}(\mathbb C)$, associated to an array $\xi\in M_N(\mathbb C^K)$ is given by 
$$T_p(\xi)_{i_1\ldots i_p,j_1\ldots j_p}=\frac{1}{K}<\xi_{i_1j_1},\xi_{i_2j_2}><\xi_{i_2j_2},\xi_{i_3j_3}>\ldots\ldots<\xi_{i_pj_p},\xi_{i_1j_1}>$$
with the scalar product being the usual one on $\mathbb C^K$, taken linear at right.
\end{proposition}

\begin{proof}
We have the following well-known computation, valid for any vectors $\xi_1,\ldots,\xi_p$ having norms $||\xi_i||\in\{0,1\}$, with the scalar product being linear at right:
\begin{eqnarray*}
&&
Proj(\xi_i)x=<\xi_i,x>\xi_i,\forall i\\
&\implies&
Proj(\xi_1)\ldots Proj(\xi_p)(x)=<\xi_1,\xi_2>\ldots\ldots<\xi_{p-1},\xi_p><\xi_p,x>\xi_1\\
&\implies&Tr(Proj(\xi_1)\ldots Proj(\xi_p))=<\xi_1,\xi_2>\ldots\ldots<\xi_{p-1},\xi_p><\xi_p,\xi_1>
\end{eqnarray*}

Thus, the matrices $T_p$ from Proposition 2.2 can be computed as follows:
\begin{eqnarray*}
(T_p)_{i_1\ldots i_p,j_1\ldots j_p}
&=&\int_Xtr\left(Proj(\xi_{i_1j_1}^x)Proj(\xi_{i_2j_2}^x)\ldots Proj(\xi_{i_pj_p}^x)\right)dx\\
&=&\frac{1}{K}\int_X<\xi_{i_1j_1}^x,\xi_{i_2j_2}^x><\xi_{i_2j_2}^x,\xi_{i_3j_3}^x>\ldots\ldots<\xi_{i_pj_p}^x,\xi_{i_1j_1}^x>dx\\
&=&\int_X(T_p(\xi^x))_{i_1\ldots i_p,j_1\ldots j_p}dx
\end{eqnarray*}

We therefore obtain the formula in the statement. See \cite{bb2}, \cite{bne}.
\end{proof}

Now back to our questions, a number of interesting properties of $\Gamma$ can be recovered from the explicit knowledge of the normalized spectral measure $\mu=law(\chi/N)$ of the main character $\chi=\sum_iu_{ii}$. Recall that this is the probability measure on the real line whose $m^{th}$ moment is 
\begin{equation*}
  \int_G\left(\frac{\chi}{N}\right)^m. 
\end{equation*}

\begin{proposition}
Given a closed subgroup $G\subset O_N^+$, consider the normalized spectral measure $\mu$ of the main character $\chi$. We then have:
\begin{enumerate}
\item $\mu$ is a real probability measure, supported on $[-1,1]$.

\item Kesten criterion: $\Gamma=\widehat{G}$ is amenable precisely when $1\in Supp(\mu)$.

\item If $G$ is finite, its cardinality $|G|=\dim_\mathbb CC(G)$ is given by $|G|=\frac{1}{\mu(1)}$.
\end{enumerate}
\end{proposition}

\begin{proof}
All these results are well-known, the idea being as follows:

(1) This is clear from $u_{ii}=u_{ii}^*$, and from $||u_{ii}||\leq1$ for any $i$, because these conditions tell us that the operator $\chi/N$ is self-adjoint, and of norm $\leq1$.

(2) This is indeed the quantum version of the Kesten criterion \cite{kes}, the idea being that $1\in Supp(\mu)$ is equivalent to having a factorization of the counit $\varepsilon:C^*_{red}(\Gamma)\to\mathbb C$.

(3) This is well-known too. If we denote by $F$ the principal graph, with adjacency matrix $A\in M_M(0,1)$, where $M=|F|$, and Perron-Frobenius vector $\xi\in\mathbb R^M$, we have:
$$\mu(1)=\lim_{p\to\infty}\frac{(A^p)_{11}}{N^p}=\frac{\xi_1^2}{||\xi||^2}=\frac{1}{\sum_r\dim(r)^2}=\frac{1}{|G|}$$

Here, and in the above two proofs as well, we have used a number of standard facts, and we refer to \cite{bmt}, \cite{ntu} for more details on all this material.
\end{proof}

Regarding now the explicit computation of $\mu$, a certain moment formula comes by putting together Proposition 2.1, Proposition 2.2, Proposition 2.3. However, as explained in \cite{bb2}, \cite{bne}, one can do better than that, with a more conceptual result, as follows:

\begin{theorem}
Given an inner faithful quasi-flat model $\pi:C(G)\to M_K(C(X))$, mapping $u_{ij}\to Proj(\xi_{ij}^x)$ with $||\xi_{ij}^x||\in\{0,1\}$, the law of the normalized character $\chi/K$ with respect to the truncated integral $\int_G^r$ coincides with that of the Gram matrix of the vectors
$$\xi_{i_1\ldots i_r}^x=\frac{1}{\sqrt{K}}\cdot\xi^{x_1}_{i_1i_2}\otimes\xi^{x_2}_{i_2i_3}\otimes\ldots\otimes \xi^{x_r}_{i_ri_1}$$
with respect to the normalized matrix trace, and to the integration functional on $X^r$.
\end{theorem}

\begin{proof}
This was proved in \cite{bb2}, \cite{bne} under various supplementary assumptions on the model, which are actually not needed. First of all, by using Proposition 2.2 above, the moments $C_p$ of the measure that we are interested in are given by:
$$C_p=\frac{1}{K^p}\int_G^r\left(\sum_iu_{ii}\right)^p=\frac{1}{K^p}\sum_{i_1\ldots i_p}(T_p^r)_{i_1\ldots i_p,i_1\ldots i_p}=\frac{1}{K^p}\cdot Tr(T_p^r)$$

The trace on the right is given by the following formula:
$$Tr(T_p^r)=\sum_{i_1^1\ldots i_p^r}(T_p)_{i_1^1\ldots i_p^1,i_1^2\ldots i_p^2}\ldots\ldots(T_p)_{i_1^r\ldots i_p^r,i_1^1\ldots i_p^1}$$

In view of the formula in Proposition 2.3, this quantity will expand in terms of the matrices $T_p(\xi)$ constructed there. To be more precise, we have:
$$Tr(T_p^r)=\int_{X^r}\sum_{i_1^1\ldots i_p^r}T_p(\xi^{x_1})_{i_1^1\ldots i_p^1,i_1^2\ldots i_p^2}\ldots\ldots T_p(\xi^{x_r})_{i_1^r\ldots i_p^r,i_1^1\ldots i_p^1}\,dx$$

By using now the explicit formula of each $T_p(\xi)$, from Proposition 2.3, we have:
\begin{eqnarray*}
Tr(T_p^r)
&=&\frac{1}{K^r}\int_{X^r}\sum_{i_1^1\ldots i_p^r}<\xi_{i_1^1i_1^2}^{x_1},\xi_{i_2^1i_2^2}^{x_1}>\ldots\ldots<\xi_{i_p^1i_p^2}^{x_1},\xi_{i_1^1i_1^2}^{x_1}>\\
&&\hskip52.8mm\ldots\\
&&\hskip22.4mm<\xi_{i_1^ri_1^1}^{x_r},\xi_{i_2^ri_2^1}^{x_r}>\ldots\ldots<\xi_{i_p^ri_p^1}^{x_r},\xi_{i_1^ri_1^1}^{x_r}>dx
\end{eqnarray*}

By changing the order of the summation, we can write this formula as:
\begin{eqnarray*}
Tr(T_p^r)
&=&\frac{1}{K^r}\int_{X^r}\sum_{i_1^1\ldots i_p^r}<\xi_{i_1^1i_1^2}^{x_1},\xi_{i_2^1i_2^2}^{x_1}>\ldots\ldots<\xi_{i_1^ri_1^1}^{x_r},\xi_{i_2^ri_2^1}^{x_r}>\\
&&\hskip52.8mm\ldots\\
&&\hskip19.4mm<\xi_{i_p^1i_p^2}^{x_1},\xi_{i_1^1i_1^2}^{x_1}>\ldots\ldots<\xi_{i_p^ri_p^1}^{x_r},\xi_{i_1^ri_1^1}^{x_r}>dx
\end{eqnarray*}

But this latter formula can be written as follows:
\begin{eqnarray*}
Tr(T_p^r)
&=&K^{p-r}\int_{X^r}\sum_{i_1^1\ldots i_p^r}\frac{1}{K}<\xi_{i_1^1i_1^2}^{x_1}\otimes\ldots\otimes\xi_{i_1^ri_1^1}^{x_r}\ ,\ \xi_{i_2^1i_2^2}^{x_1}\otimes\ldots\otimes \xi_{i_2^ri_2^1}^{x_r}>\\
&&\hskip55mm\ldots\\
&&\hskip15mm\frac{1}{K}<\xi_{i_p^1i_p^2}^{x_1}\otimes\ldots\otimes \xi_{i_p^ri_p^1}^{x_r}\ ,\ \xi_{i_1^1i_1^2}^{x_1}\otimes\ldots\otimes\xi_{i_1^ri_1^1}^{x_r}>dx
\end{eqnarray*}

In terms of the vectors in the statement, and of their Gram matrix $G_r^x$, we obtain:
\begin{eqnarray*}
Tr(T_p^r)
&=&K^{p-r}\int_{X^r}\sum_{i_1^1\ldots i_p^r}<\xi_{i_1^1\ldots i_1^r}^x,\xi_{i_2^1\ldots i_2^r}^x>\ldots\ldots<\xi_{i_p^1\ldots i_p^r}^x,\xi_{i_1^1\ldots i_1^r}^x>dx\\
&=&K^{p-r}\int_{X^r}\sum_{i_1^1\ldots i_p^r}(G_r^x)_{i_1^1\ldots i_1^r,i_2^1\ldots i_2^r}\ldots\ldots(G_r^x)_{i_p^1\ldots i_p^r,i_1^1\ldots i_1^r}\,dx\\
&=&K^{p-r}\int_{X^r}Tr((G_r^x)^p)dx
\end{eqnarray*}

Summarizing, the moments of the measure in the statement are given by:
$$C_p=\frac{1}{K^r}\int_{X^r}Tr((G_r^x)^p)dx=\left(tr\otimes\int_{X^r}\right)\left(G_r^p\right)$$

This gives the formula in the statement of the theorem.
\end{proof}

As a conclusion, various properties of $\Gamma$ can be recovered by plugging the Gram matrix law in Theorem 2.5, via a Ces\`aro limiting procedure as in Proposition 2.1, into the general criteria from Proposition 2.4. For some applications of this method, see \cite{bb2}, \cite{bne}.

In principle, our ``capturing'' philosophy should have as well some other applications. The most interesting questions are perhaps those related to the growth of $\Gamma$:

\begin{definition}
Given $G\subset U_N^+$, with fundamental corepresentation satisfying $1\in u\sim\bar{u}$, the growth function of its dual $\Gamma=\widehat{G}$ is the series $f(z)=\sum_{n\geq0}v_nz^n$, where
$$v_n=\sum_{r\in Irr(G),|r|\leq n}\dim(r)^2$$
with the ``length'' function being defined as $|r|=\inf\left\{l\in\mathbb N|r\in u^{\otimes l}\right\}$.
\end{definition}

As an illustration, given a discrete group $\Gamma=<g_1,\ldots,g_N>$, we can enlarge if needed the set $S=\{g_1,\ldots,g_N\}$, as to have $1\in S=S^{-1}$. Thus, we obtain an embedding $\widehat{\Gamma}\subset U_N^+$, with $u=diag(g_1,\ldots,g_N)$ satisfying $1\in u\sim\bar{u}$. We have then $Irr(\widehat{\Gamma})=\Gamma$, and the length function is the usual one on $\Gamma$, with respect to $S$. Thus, the above numbers $v_n$ are the volumes of the corresponding balls, and $f$ is their generating series. See \cite{dpr}.

It is quite unclear on how to capture the growth, from the knowledge of an inner faithful model. However, some evidence for this comes from Gromov's result in \cite{gro}, stating that polynomial growth is equivalent to being virtually nilpotent, and from the related work in \cite{dpr}. More precisely, the recurrence of the random walk on a finitely generated discrete group is equivalent to quadratic growth. In turn, Gromov's theorem \cite{gro} implies that such groups are finite extensions of $\mathbb{Z}^n$ for $n\le 2$ (see e.g. \cite[$\S$VI.6]{var}) 

This circle of ideas thus suggests strong connections between growth and the random walk invariants, and the latter fall into the class of quantities that we can ``recapture'' by using our methods. 

Another interesting question is that of recapturing the diagonal quotient of $\Gamma$, whose dual is the diagonal subgroup of $G=\widehat{\Gamma}$, from the knowledge of an inner faithful model. Once again, we cannot quite expect here to have exact results, but rather asymptotic ones. Some interesting work on a number of related topics was recently done in \cite{fls}.

Summarizing, our opinion on these questions would be as follows: 

\begin{heuristic}
The following properties of a discrete quantum group $\Gamma$ can be recaptured via analytic methods, from the knowledge of an inner faithful model for $C^*(\Gamma)$:
\begin{enumerate}
\item The asymptotic behavior of the growth invariants.

\item The asymptotic behavior of the random walk on the diagonal quotient.
\end{enumerate}
\end{heuristic}

Finally, by exiting now the inner faithfulness setting, one interesting question is whether the property of being residually finite from \cite{chi} has or not a probabilistic formulation, in terms of some associated universal matrix models. Once again, having such a result would be probably very useful, but for the moment, we have no idea here.

\section{Stationarity questions}

We discuss now stationarity questions for the universal quasi-flat models of the group duals. To be more precise, we consider uniform groups $\mathbb Z_K^{*M}\to\Gamma\to\mathbb Z_K^M$ which are virtually abelian, and our aim is that of computing the associated model space $X_G$, and then proving that the universal model is stationary. This question is quite interesting, because we will have here a substantial improvement of Thoma's theorem \cite{tho}.

The first question that we study is the computation of the model space. We would like for instance to understand if this space splits as a union of homogeneous spaces.

As a first remark, in the finite case we have the following result:

\begin{proposition}
When the intermediate quotient $\mathbb Z_K^{*M}\to\Gamma\to\mathbb Z_K^M$ is finite, $X_G$ is a union of homogeneous spaces, for certain actions of the unitary group $U_K$.
\end{proposition}

\begin{proof}
This follows indeed from the fact that the space of $K$-dimensional unitary representations of $\Gamma$ is discrete, and so when two representations are sufficiently close, they must belong to the same orbit, under the conjugation action of $U_K$.
\end{proof}

We now specialize to two-generator finite homogeneous groups, examining first the case when the generators are involutions. These groups are all well-known, as follows:

\begin{proposition}
The uniform groups with two order $2$ generators are:
\begin{enumerate}
\item The dihedral groups $D_n$, with $n$ even.

\item The infinite dihedral group $D_\infty$.
\end{enumerate} 
\end{proposition}

\begin{proof}
We use the canonical identification $D_\infty=\mathbb Z_2*\mathbb Z_2$. A group $\Gamma=<g_1,g_2>$ as above fits into a sequence that identifies the pairs of respective generators, as follows:
$$D_\infty\to\Gamma\to\mathbb Z_2^2$$

By using now the canonical identification $D_\infty=\mathbb Z\rtimes\mathbb Z_2$, the normal subgroups of $D_\infty$ are the subgroups $n\mathbb Z\subset\mathbb Z\subset \mathbb Z\rtimes\mathbb Z_2$ with $n\in\mathbb N\cup\{\infty\}$, and these produce the quotients $D_n=\mathbb Z_n\rtimes\mathbb Z_2$. Furthermore, such a quotient surjects onto $\mathbb Z_2^2$ by sending the generators of $\mathbb Z_n,\mathbb Z_2$ to non-trivial distinct elements of $\mathbb Z_2^2$ precisely when $n$ is even. 

Finally, the fact that the groups that we found, namely $D_n$ with $n\in 2\mathbb N\cup\{\infty\}$, are all uniform is clear. Indeed, such a dihedral group $D_n$ admits an automorphism interchanging the generator of $\mathbb Z_2$ and the generator of the other copy of $\mathbb Z_2$.
\end{proof}

The case $D_\infty=\mathbb Z_2*\mathbb Z_2$ being discussed in \cite{bfr}, we restrict now the attention to the finite case. Our result here is as follows:

\begin{proposition}
Let $\Gamma=D_n$ with $n$ even, and set $G=\widehat{\Gamma}$. Then:
\begin{enumerate}
\item $X_G$ is a union of homogeneous spaces $X_+$, $X_-$ and $X_\chi$ labeled by the $\frac{n}{2}-1$ characters of the two-dimensional representations of $D_n$.

\item The model is stationary with respect to the probability measure that is uniform on each component $X_\bullet$ and assigns equal weights to the $X_\chi$ and $X_+\cup X_-$.
\end{enumerate}
\end{proposition}

\begin{proof}
We use the fact that the points of $X_G$ are the representations $\phi:D_n\to U_2$. These representations fall into two classes, as follows:

-- Direct sums of two 1-dimensional representations.
    
-- Irreducible 2d representations, indexed by the corresponding $\frac{n}{2}-1$ characters.

With this description in hand, both the assertions follow:

(1) The direct sums fall into two subcases, depending on whether the unitaries $\phi(g_i)$ are equal or differ by a sign, and these give rise to two connected components of $X_G$, denoted $X_+,X_-$. These two spaces are both homogeneous under the $U_2$ conjugation action, and the isotropy groups of points in these sets are maximal tori in $U_2$.

As for the irreducible 2d representations, each of them corresponds to a connected component $X_\chi$ of $X_G$, which is a homogeneous space under the action of the unitary group $U_2$ by conjugation. The isotropy group of each point in $X_\chi$ is the center of $U_2$.

Summarizing, we have obtained the description of $X_G$ given in the statement.

(2) We recall from \cite{ser} that for a finite group $\Gamma$, the following holds, for any $g\neq1$:
$$\sum_{\pi\in Irr(G)}\mathrm{dim}(\pi)\chi_\pi(g)=0$$

In our case this formula is as follows, with the first sum being over the characters $\chi$ as above, and with $\xi_i$ ranging over the four 1-dimensional representations of $D_n$:
$$2\sum_\chi\chi(g)+\sum_i\xi_i(g)=0$$

But this is exactly the stationarity formula claimed in the statement, applied to an  arbitrary element $g\in\Gamma-\{1\}$, with respect to the weights indicated there.
\end{proof}

Observe that even though $X_\pm$ are disjoint, the union $X_+\cup X_-$ is itself a homogeneous space, since its two components are (non-canonically) isomorphic.

We now consider two-generator homogeneous groups with higher order generators. The situation here can be more complicated, as shown by the following result:

\begin{theorem}
Consider the Heisenberg group of order $K^3$, namely
$$\Gamma=\left<g_1,g_2\Big|g_1^K=g_2^K=[g_1,g_2]^K=1,[g_1,g_2]={\rm central}\right>$$
with $K\in\mathbb N$ assumed to be prime. Then the following hold:
\begin{enumerate}
\item The universal model space $X_G$ is not homogeneous.

\item However, $X_G$ is a union of $K!+K-1$ connected homogeneous spaces. 

\item The universal model is stationary, with respect to some suitably chosen weights. 
\end{enumerate}
\end{theorem}

\begin{proof}
  Consider the universal representation $\pi:\Gamma\to C(X_G,U_K)$ and fix a primitive $K$-th root of unity $w$. because $K$ is prime, the irreducible representations of the Heisenberg group are either characters of the quotient $\Gamma/\langle [g_1,g_2]\rangle$ or irreducible and $K$-dimensional. It follows that for each $x\in X_G$ the operator $\pi_x([g_1,g_2])$ is a scalar and hence a power of $w$.  

  For each $0\leq l\leq K-1$ we get a subspace $X_l\subset X_G$ defined as the collection of those representations $\pi_x$ for which $\pi_x([g_1,g_2])=w^l$. We then have:
$$X_G=X_0\sqcup\cdots\sqcup X_{K-1}$$

With this decomposition in hand, we can now prove our results:

(1,2) Our first claim is that the spaces $X_l$ defined above are smooth real manifolds, having dimensions as follows:
\begin{eqnarray*}
\dim_\mathbb R(X_0)&=&K(K-1)\\
\dim_\mathbb R(X_l)&=&(K+1)(K-1)\ {\rm for}\ l>0
\end{eqnarray*}

In order to prove this claim, let us first look at $X_0$. This space consists of $K$-dimensional representations which factor through the quotient $\Gamma\to\mathbb Z_K^2$, and for which the two generators have eigenvalues $w^i$ for $0\le i\le K-1$. But such a representation is specified by the data consisting of the $w^i$-eigenspaces $V_i$ of $g_1$, and the eigenvalues $w^{\tau(i)}$ of $g_2$ on the $V_i$, where $\tau$ is a permutation of the set $\{0,\ldots,K-1\}$.

Thus $X_0$ breaks up as a disjoint union of $K!$ components, indexed by the permutations $\tau\in S_K$. Each component, in turn, is isomorphic to the space $T_K$ of ordered $K$-tuples of orthogonal lines in $\mathbb C^K$ and is thus homogeneous under the action of $U_K$ on such $K$-tuples. Choosing a single line in $\mathbb C^K$ exhibits $T_K$ as a bundle over $P^{K-1}_\mathbb C$ with fiber $T_{K-1}$, so by induction the real dimension of $X_0$ follows to be, as claimed:
$$\dim_\mathbb R(X_0)=2(k-1)+2(k-2)+\cdots = K(K-1)$$

Let us discuss now the case $l>0$. Here all the representations making up $X_l$ are irreducible and mutually isomorphic, so $X_l$ is a homogeneous space under the action of $U_K$ by conjugation. Moreover, the isotropy group of a point in $X_l$, identified with the corresponding representation, is the center $\mathbb{T}\subset U_K$, so we have, as claimed:
$$\dim_\mathbb R(X_l) = \mathrm{dim}_{\mathbb{R}}(U_K)-1 = K^2-1$$

All in all, we have the $K!$ homogeneous connected components that make up $X_0$ and the $K-1$ connected homogeneous components $X_l$, $l\ge 1$, proving the first assertion. 

(2) This is very similar to the proof of the analogous assertion in Proposition 3.3 above. Consider indeed the uniform probability measures $\mu_l$ on $X_l$ respectively, with ``uniform'' meaning by definition invariant under the action of $U_K$, and in the case of $X_0$, assigning equal masses to the $K!$ connected components. The probability measure on $X_G$ that will give us the stationarity will be then the average of the measures $\mu_l$. 

In order to prove that we have indeed the stationarity property, let $\phi_l$ be the normalized traces on $\Gamma$ attached to the spaces $X_l$. At $l=0$ we have the following formula, where the sum ranges over the $K^2$ characters $\xi$ of the quotient $\Gamma\to\mathbb Z_K^2$:
$$\phi_0(g) = \frac{(K-1)!}K \sum_\xi \xi(g),\ \forall g\in\Gamma$$

At $l\geq1$ now, we have the following formula, where $\chi_l$ is the character corresponding to the $K$-dimensional irreducible representations that $X_l$ consists of:
$$\phi_l=\frac{\chi_l}{K}$$

By integrating now, at $l=0$ we have the following formula: 
$$\left(tr\otimes\int_{X_G}(.)\,d\mu_0\right)\pi=\frac{1}{K!}\,\phi_0$$

As for the $l\geq1$ case, here the formula is as follows:
$$\left(tr\otimes\int_{X_G}(.)\,d\mu_l\right)\pi=\phi_l$$

In order now to finish, we can use the following formula, which is analogous to the one that we used in the proof of Proposition 3.3 above:
$$\sum_{\xi\in\widehat{\mathbb{Z}_K^2}}\xi(g) + K\sum_{\ell\ge 1}\chi_l(g)=0,\ \forall g\in\Gamma-\{1\}$$

Indeed, consider the average of the measures on the various components of $X_G$:
$$\mu=\frac{1}{K}\sum_{i=0}^{l-1}\mu_l$$

By the above formulae, $\mu$ makes our model stationary, and we are done.
\end{proof}

\section{Metabelian groups}

We discuss now stationarity questions for general virtually abelian uniform groups. As explained in \cite{bch}, it follows from Thoma's theorem \cite{tho} that the corresponding group algebras have stationary models, of a certain special type, and our aim here is that of proving some finer results of this type, involving this time quasi-flat models.

Consider a group $\mathbb Z_K^{*M}\to\Gamma\to\mathbb Z_K^M$ which is uniform, in the sense that the symmetric group $S_M$ acts on the generators. Observe that $\Gamma$ must be the quotient of $\mathbb Z_K^{*M}$ by a $S_M$-invariant normal subgroup contained in the derived subgroup $(\mathbb Z_K^{*M})'$. We have:

\begin{proposition}
The following hold:
\begin{enumerate}
\item If $\Gamma_1,\ldots,\Gamma_M$ are abelian, the derived subgroup of $\Gamma=\Gamma_1*\ldots*\Gamma_M$ is free.

\item If $\Gamma_i$ are all finite, the derived subgroup is free on finitely many generators.

\item If $\Gamma$ is uniform, the kernel of $\mathbb Z_K^{*M}\to \Gamma$ is a free subgroup of the domain.
\end{enumerate}
\end{proposition}

\begin{proof}
This follows indeed by using the basic theory of groups:

(1) Let $H$ be a subgroup of the free product $\Gamma$ in the statement. According to the Kurosh theorem, this subgroup decomposes as follows, with $F\subset\Gamma$ being free, and with the groups $H_i$ being conjugate to subgroups in various free factors $\Gamma_i\subset\Gamma$:
$$H=F*H_1*\ldots*H_l$$

By using the above subgroups $\Gamma_i$, we can consider the following surjection:
$$\Gamma\to\Gamma_1\times\ldots\times\Gamma_M$$

The non-trivial $H_i$ will have a non-trivial image through it, whereas the commutator subgroup $\Gamma'\subset\Gamma$ is annihilated by the homomorphism. It follows that there are no additional free factors $H_i$ in the decomposition of $H$, i.e. the latter is free.

(2) This follows by using Schreier's lemma, which states that a finite index subgroup of a finitely generated group is again finitely generated. 

(3) We know that the kernel is a subgroup of the derived group $\left(\mathbb Z_K^{*M}\right)'$. But the latter is free by using (2) above, and the conclusion follows from the Nielsen-Schreier theorem, which states that the subgroups of the free groups are free.
\end{proof}

We recall that a metabelian group is a group $\Gamma$ whose commutator subgroup is abelian. Equivalently, $\Gamma$ must appear as an extension of an abelian group by another abelian group. As basic examples, we have for instance the dihedral groups, discussed in section 3.

In our uniform group setting, we can formulate the following definition: 

\begin{definition}
Consider a uniform group $\mathbb Z_K^{*M}\to\Gamma\to\mathbb Z_K^M$.
\begin{enumerate}
\item We say that $\Gamma$ is metabelian when the kernel of $\Gamma\to\mathbb Z_K^M$ is abelian.

\item We set $\left(\mathbb Z_K^{*M}\right)_{met}=\mathbb Z_K^{*M}/\left(\mathbb Z_K^{*M}\right)''$, and call it universal metabelian.
\end{enumerate}
\end{definition}

Observe that our notion in (1) agrees with the usual definition of the metabelian groups, given above. Regarding now (2), observe that the quotient there is indeed uniform, because the group we are quotienting out is characteristic, and hence invariant under the action of $S_M$ on $\mathbb Z_K^{*M}$. Thus the $S_M$-action descends to $\left(\mathbb Z_K^{*M}\right)_{met}$, as required.

We have the following results, regarding such groups:

\begin{proposition}
The universal metabelian uniform group is as follows:
\begin{enumerate}
\item This is an extension of $\mathbb Z_K^M$ by a free abelian group of finite rank.

\item At $M=2$, this is an extension of $\mathbb Z_K^2$ by a free abelian group of rank $(K-1)^2$.
\end{enumerate}
\end{proposition}

\begin{proof}
The first assertion follows from Proposition 4.1. Regarding the second assertion, set $\Omega=\mathbb{Z}_K^{*2}$, with standard generators denoted $x,y$. According to the Schreier lemma, a set of generators for its derived subgroup $\Omega'$, which is free by the above results, can be obtained as follows. First, consider the set of representatives for the cosets of $\Omega'$:
 $$R=\left\{x^iy^j\Big|0\leq i,j\leq K-1\right\}$$
 
If for each element $g\in\Omega$ we denote by $\overline{g}$ the representative of the coset of $g$, then the generating set is as follows:
$$\Omega'=\left<gx\left(\overline{gx}\right)^{-1}\Big|g\in R\right>$$
 
Since only those $g=x^iy^j$ with $j\geq 1$ are of relevance, this leaves us with the $K(K-1)$ generators obtained by conjugating by all $x^i$, $0\le i\le K-1$, the following elements:
$$[y^j,x],\ 1\leq j\leq K-1$$

These elements are still not independent, because for each $1\leq j\leq K-1$ the product of the elements $x^i[y^j,x]x^{-i}$, with $0\leq i\leq K-1$,   in the order of increasing indices $i$ equals the inverse of $[y^j,x]$. Thus, we are led to the following system of generators:
$$g_{ij}:=x^i[y^j,x]x^{-i},\ 1\le j\le K-1,\ 1\le i\le K-1$$
It remains to argue that these generators indeed form a free family. In fact, they already do so in $\mathbb{Z}_K^{*2}$ (rather than its metabelian quotient). To verify this, consider a reduced product $g$ of elements $g_{ij}^{\pm 1}$, in the sense that no factor $g_{ij}$ appears next to a factor $g_{ij}^{-1}$. It is then easy to see that further extending $g$ to a reduced product $gg_{ij}^{\pm 1}$ does not decrease the number of factors $y^j$, $j\in \mathbb{Z}\setminus\{0\}$ in the expansion
\begin{equation*}
  g=x^{a_1}y^{b_1}x^{a_2}\cdots
\end{equation*}
as a word in the letters $x^i$ and $y^j$ for $1\le i,j\le K-1$. Since such a word is non-trivial as an element of $\mathbb{Z}_K^{*2}$ as soon as it is non-empty, the conclusion follows. 
\end{proof}

Now consider the universal group $\Gamma=(\mathbb Z _K^{*2})_{met}$ discussed above. Its commutator is, according to the above result, a free abelian group of rank $(K-1)^2$. We have:

\begin{proposition}
We can identify the derived group $\Gamma'$ with the tensor square $(\mathbb{Z}^{K-1})^{\otimes 2}$ such that the actions of the generators $x,y\in \Gamma$ on $\Gamma'$ are via the matrix
$$\begin{pmatrix}
0&\cdots&0&-1\\
1&\cdots&0&-1\\
\vdots&\ddots&\vdots&\vdots\\
0&\cdots&1&-1
\end{pmatrix}$$
on the two respective tensor components $\mathbb Z^{K-1}$.
\end{proposition}

\begin{proof}
This is indeed clear from the above results.
\end{proof}

We examine now in detail the quotient $\Gamma$ of the group $(\mathbb Z_3^{*2})_{met}$ by the subgroup consisting of the squares of the elements of the commutator subgroup $(\mathbb Z_3^{*2})_{met}'\simeq \mathbb Z^4$. We know that $\Gamma$ is an extension of $\mathbb Z_3^2$ by $\mathbb Z_2^4$, and hence has order $16\cdot 9=144$. We will now describe the universal model space $X_G$ associated to the compact dual $G=\widehat{\Gamma}$.

First, $X_G$ consists by definition of 3-dimensional representations of $\Gamma$. Since $A=\Gamma'$ is abelian, the general theory in \cite{ser} shows that the dimension of every irreducible representation divides the order of $\Gamma/A$, which is $9$. In conclusion, a 3-dimensional representation is either irreducible or a sum of three 1-dimensional representations. Moreover, 1-dimensional representations are automatically trivial on $A$, and hence are characters of $\Gamma_{ab}\simeq\mathbb Z_3^2$. Based on this description, let us introduce:

\begin{definition}
Let $\Gamma$ and $X_G$ be as above.
\begin{enumerate}
\item We say that a component of $X_G$ is solid if is an orbit under the conjugation by $U_3$ of a $3$-dimensional irreducible representation of $\Gamma$.

\item Also, we say that a component of $X_G$ is loose if it consists of representations that break up as sums of $3$ irreducible representations of $\Gamma_{ab}$.
\end{enumerate}
\end{definition}

The first remark about the structure of $X_G$ is as follows:

\begin{proposition}
The space $X_G$ has $6$ loose components $\{X_\tau|\tau\in S_3\}$, which are all isomorphic as $U_3$-spaces to $U_3/\mathbb T^3$, where $\mathbb T^3\subset U_3$ is the subgroup of diagonal matrices.
\end{proposition}

\begin{proof}
Our space $X_G$ coincides with the model space associated to $\mathbb Z_3^2$, and consists of commuting pairs of unitary matrices with eigenvalues $1,w,w^2$, where $w=e^{2\pi i/3}$.

Two such commuting matrices $x,y$ must have coinciding eigenspaces, and we therefore have 6 components, corresponding to the 6 permutations of 3 symbols.
\end{proof}

With notations from the above proof, we agree to denote by $X_\tau$ the component of $X_G$ corresponding to the following function, identified with $\tau\in S_3$:
$$\{w^i\}=\text{ spectrum of }x\to \text{ spectrum of y}=\{w^i\}$$

Regarding now the solid components of $X_G$, these correspond to
the 3-dimensional irreducible representations of $\Gamma$, as noted above. In order to describe now these representations, we use the general theory in \cite{ser}. We fix a subgroup $\mathbb Z_3^2\simeq H\subset G$ that maps bijectively onto $G/A$, for instance any Sylow $3$-subgroup of $G$. We may as well assume $x\in H$, and we denote by $z$ a second generator. Now, $H$ acts by conjugation on $A$ and hence also on its character group $\widehat{A}$, and with this picture in mind, we have:

\begin{proposition}
The irreducible representations of $\Gamma$ are labeled by pairs $(O,\rho)$, where:
\begin{enumerate}
\item $O$ is an orbit of the $H$-action on the character group $\widehat{A}$.

\item Having fixed an element $\chi=\chi_O\in O$ for each such orbit $O$, we denote by $H_\chi$ the isotropy group of $\chi$ under this action.

\item $\rho$ ranges over the irreducible representations of $H_\chi$.
\end{enumerate}
\end{proposition}

\begin{proof}
This is indeed clear from \cite{ser}, with the representation associated to $(O,\rho)$ being obtained by first extending the character $\chi=\chi_O$ to $AH_\chi$, then constructing the representation $\chi\otimes\rho$ of this same group, and then inducing up to $G=AH$.
\end{proof}

We conclude that the 3-dimensional irreducible representations of $\Gamma$ correspond to the pairs $(O,\rho)$, where $O$ is a size $3$ orbit of the action of $H\simeq\mathbb Z_3^2$ on $\widehat{A}\simeq\mathbb Z_2^4$ and
$\rho$ is a character of the isotropy group, which is isomorphic to $\mathbb Z_3$, of a fixed character in $O$.

On the other hand, we have as well the following result:

\begin{proposition}
The group $\widehat{A}$ breaks up under the action of $H$ into four orbits, as follows:
\begin{enumerate}
\item A singleton, consisting of the trivial character.

\item A size $3$ orbit $O_1$, consisting of the characters fixed by $xy$.

\item A size $3$ orbit $O_2$, consisting of the characters fixed by $xy^2$.
  
\item All the other characters, making up a single size $9$ orbit.
\end{enumerate}
\end{proposition}

\begin{proof}
This follows indeed from the explicit description of the action of $H$ on $A$, given in our results above.
\end{proof}

Now observe that for all three-dimensional irreducible representations corresponding to the cases (b,c), by \cite{ser} these have traces zero on $x,y$. It follows that these representations assign spectrum $\{w^i\}$ to $x,y$, and hence are contained in $X_G$, and so:

\begin{proposition}
The space $X_G$ has $6$ solid components isomorphic to $PU_3$, each consisting of irreducible representations of $\Gamma$ attached to one of the pairs $(i,\rho)$, where:
\begin{enumerate}
\item $O_i$ consists of the $A$-characters fixed by $xy^i$, $i=1,2$.

\item $\rho$ is one of the three characters of $\langle xy^i\rangle\simeq\mathbb Z_3$.
\end{enumerate}
We agree to denote the respective components of $X_G$ by $X_{i,\rho}$.
\end{proposition}

\begin{proof}
This follows indeed from the above discussion.
\end{proof}

Summarizing, $X_G$ has 12 components, falling into two different
classes. With these ingredients in hand, we can now prove the following negative result:

\begin{theorem}
The canonical representation $\pi:\Gamma\to M_3(C(X_G))$ cannot be stationary with respect to a measure of type
$$\mu=\sum_{\tau\in S_3}\alpha_\tau\mu_{X_\tau} +\sum_{i\rho}\beta_{i,\rho} \mu_{X_{i,\rho}}$$
and this, for any choice of scalars $\alpha_{\bullet},\beta_{\bullet}\geq 0$ summing up to $1$.
\end{theorem}

\begin{proof}
We recall that we are denoting by $H\simeq\mathbb Z_3^2$ a
complement of $A$ in $\Gamma$, i.e. a subgroup of $\Gamma$ that maps
isomorphically onto $G/A$. We also assume that $x\in H$ and fix another generator $z\in H$ in the same class as $y$ modulo $A$. Now denote:
$$\alpha=\sum_\tau\alpha_\tau,\ \beta_i=\sum_\rho\beta_{i,\rho}\text{ for }i=1,2$$

For any $a\in A$ the contribution of the loose components to the normalized trace of $\pi(a)$ is $\alpha$, while the contribution of the solid components $X_{i,\rho}$ for $\rho$ ranging over $\widehat{\langle xz^i\rangle}$ is:
$$\beta_i\sum_{\chi\in o_i}\chi(a)$$

We identify $\widehat{A}\simeq(\mathbb{F}_2^2)^{\otimes 2}$, and fix bases $\{p_i\}$ and $\{q_i\}$ for the two tensor components with respect to which the conjugation by $x,y$ acts on $A$ as $m\otimes I_2,I_2\otimes m$, where:
$$m=\begin{pmatrix}0&1\\1&1\end{pmatrix}$$
 
It follows from this choice and a simple computation that the above expressions, in the order $i=1$ and $i=2$, are:
\begin{enumerate}
\item $-1$ and $-1$ for $a=p_1\otimes q_1$.

\item $3$ and $-1$ for $a=p_1\otimes q_1+p_2\otimes q_2$.

\item $-1$ and $3$ for $a=p_1\otimes q_2+p_2\otimes q_1$.
\end{enumerate}

In order for the stationarity property to hold, we must find variables $\alpha,\beta_1,\beta_2$ satisfying $\alpha+\beta_1+\beta_2=1$ and:
$$\alpha-\beta_1-\beta_2 = \alpha+3\beta_1-\beta_2 = \alpha-\beta_1+3\beta_2=0$$

But this is impossible, and we are done.
\end{proof}

\section{Inner faithfulness}

We specialize now to the case of the metabelian groups $\Gamma$ with $M=2$, with the two generators denoted $x,y$. Such a group will be an extension of its quotient $\Gamma\to\mathbb Z_K^2$ by some abelian normal subgroup $A$, which is in turn a quotient of $(\mathbb Z_K^{*2})'$.
 
Proving that $\pi:\Gamma\to M_K(C(X_G))$ is inner faithful entails showing that for every $g\neq1$ there is some $p\in X_G$ such that $\pi_p(g)\neq1$. Note that it suffices to do this for $g\in A$. Indeed, every $g\not\in A$ acts non-trivially through some character of the quotient $\Gamma\to\mathbb Z_K^2$, and hence is non-trivial in a representation $\pi_p:\Gamma\to \mathbb{Z}_K^2\to M_K(\mathbb C)$, for some $p\in X_G$. 

We now describe the representations $\pi_p$ that will serve our purpose. We will need:

\begin{definition}
We say that a subgroup $H\simeq\mathbb Z_K$ of the quotient $\Gamma\to\mathbb Z_K^2$ is generic if it intersects neither $<x>$ nor $<y>$.
\end{definition}

Now let $\mathbb Z_K\simeq H\subset\mathbb Z_K^2$ be generic, and let
$\chi\in\widehat{A}$ be a character fixed by $H$ under the conjugation
action of $\Gamma_{ab}\simeq\mathbb Z_K^2$ on $\widehat{A}$. Then $\chi$ extends to a character of $\psi^{-1}H$, that we denote by the same symbol. We then have:

\begin{proposition}
The induced representation $\mathrm{Ind}_{\psi^{-1}(H)}^{\Gamma}(\chi)$ belongs to $X_G$.
\end{proposition}

\begin{proof}
Since $H$ is generic, both $<x>$ and $<y>$ are systems of representatives for the cosets of $\Gamma$ modulo $\psi^{-1}(H)$, and the definition of the induced representation then shows that we can find a basis for it on which $x$ (or $y$) acts as a cycle of length $K$. It follows that its eigenvalues are the $K$-th roots of unity, each with multiplicity one.
\end{proof}

In conclusion, if $1\neq a\in A$ is not trivialized by the character $\chi$ fixed by the generic subgroup $H\subset \Gamma_{ab}$, then it cannot be in the kernel of $\pi:\Gamma\to M_K(C(X_G))$. Thus:

\begin{proposition}
Suppose that the characters of $A$ whose isotropy group in $\Gamma_{ab}$
contains some generic subgroup generate $\widehat{A}$. Then $\pi:\Gamma\to M_K(C(X_G))$ is inner faithful.
\end{proposition}

\begin{proof}
This follows from the above discussion, because for every $1\neq a\in A$ there is some character $\chi$ fixed by some generic subgroup of $\Gamma_{ab}$ such that $\chi(a)\ne 1$.
\end{proof}

We can now prove a general inner faithfulness result, as follows:

\begin{proposition}
Let $K$ be a prime and $\Gamma$ a metabelian $(K,2)$-uniform group whose derived subgroup $A$ is torsion-free. Then, the canonical representation $\pi:\Gamma\to M_K(C(X_G))$ is inner faithful.
\end{proposition}

\begin{proof}
We agree to call ``generically fixed'' the characters of $A$ whose isotropy group in $\Gamma_{ab}$ contains a generic subgroup of the latter. According to the above results and to our torsion-freeness assumption, $\widehat{A}$ is a torus of dimension $\leq(K-1)^2$. Let $T\subset\widehat{A}$ be the subgroup generated by generically fixed characters. By Proposition 5.3, it suffices to prove that $T$ cannot be a proper subgroup.

Now consider the action of $\Gamma_{ab}$ on the complexified Lie algebra $V$ of $\widehat{A}$. Proposition 4.4 above shows that the generators $x,y$ both act with eigenvalues $w^i$ for for some values $1\leq i\leq K-1$, where $w=e^{2\pi i/K}$. It follows from this that for any non-trivial subspace $W\subset V$ which is invariant under $\Gamma_{ab}$ we can find $x^iy^j$ for some choice of $1\leq i,j\leq K-1$ that fixes a non-trivial vector of $W$. Applying this to a complement $W$ of the complexified Lie algebra of $T$, we conclude that $W$ must be trivial and hence $T=\widehat{A}$, as desired.
\end{proof}

We now turn to the case where the derived subgroup has torsion. We reprise our notation from above, with $A$ standing for the abelian derived subgroup of $\Gamma$.

The torsion subgroup $A_{tors}$ splits up as a direct sum of finite abelian
$q$-groups for various primes $q$. We denote these summands
by $A_q$, and call them $q$-primary components of $A$. We have the following primary analogue of Proposition 5.4 above:

\begin{proposition}
Let $K$ be a prime and $\Gamma$ a metabelian $(K,2)$-uniform group with $q$-primary derived subgroup $A$ for some prime $q\ne K$. Then, the canonical representation $\pi:\Gamma\to M_K(C(X_G))$ is inner faithful.
\end{proposition}

\begin{proof}
The proof follows the same plan as that of the analogous result for torsion-free $A$, working over finite fields rather than $\mathbb C$. Once more, we denote by $T\subset\widehat{A}$ the group generated by generically fixed characters, and seek to show that $T=\widehat{A}$.

Assume that this is not true. Then $\widehat{A}/T$ is a non-trivial $q$-group. Our claim is that a generically fixed element in $\widehat{A}/T$ lifts to a generically fixed element in $\widehat{A}$. 

To see this, consider the following short exact sequence of abelian groups:
$$1\to T\to\widehat{A}\to\widehat{A}/T\to 1$$

Let also $\mathbb Z_K$ be a group generated by some $x^iy^j$, with $1\leq i,j\leq K-1$. We have then a long exact sequence as follows, of cohomology groups over $<x^iy^j>$:
$$0\to H^0(T)\to H^0(\widehat{A})\to H^0(\widehat{A}/T)\to H^1(T)\to\ldots$$

The elements fixed by $x^iy^j$ correspond to elements of the $0$-th cohomology groups. On the other hand, note that all $H^p(T)$, $p\ge 1$ vanish because they are annihilated by both $K=|\langle x^iy^j\rangle|$ and $|T|$ (a power of $q$), and we are assuming $(K,q)=1$. The long exact sequence then implies that $H^0(\widehat{A})\to H^0(\widehat{A}/T)$, which proves our claim.  

Now with this claim in hand, and given as well our assumption that all the generically fixed elements of $\widehat{A}$ are already contained in $T$, we conclude that the predual $B\subset A$ of $\widehat{A}/T$ has no generically fixed non-trivial characters.

We can conclude analogously to the previous proof. Indeed, $x,y$ act on the $\mathbb F_q$-vector space $B/qB$ with eigenvalues $w^i$ for various $1\leq i\leq K-1$, where $w$ is a primitive $K$-th root of unity in some algebraic closure of the field $\mathbb{F}_q$ with $q$ elements.

As before, some $x^iy^j$ with $1\leq i,j\leq K-1$ fixes some nonzero element of the dual vector space $\widehat{B/qB}\subset\widehat{A}/T$. But this contradicts the non-existence of generically fixed elements in the latter group, and we are done.
\end{proof}

We can now formulate our main result, as follows:

\begin{theorem}
Let $K$ be a prime number, and $\Gamma$ a metabelian $(K,2)$-uniform group whose derived subgroup has trivial $K$-primary component. Then, the canonical representation $\pi:\Gamma\to M_K(C(X_G))$ is inner faithful.
\end{theorem}

\begin{proof}
This follows indeed by combining the previous two results.
\end{proof}

Summarizing, we have now a whole number of new results regarding the discrete group case, notably complementing those from \cite{bfr}. The unification of the present work with the one in \cite{bfr} remains a key open question, that we would like to raise here.

\section{Twisted orthogonal groups}

In this section we go back to the general framework of the discrete quantum groups, from Sections 1-2 above. There are many known inner faithfulness and stationarity results available here, notably from \cite{ba2}, \cite{bch}, \cite{bfr}, \cite{bne}, and our purpose is to bring some new contributions to the subject, by investigating some basic cocycle twists.

We recall that the standard twist $O_n^{-1}$ of the orthogonal group $O_n$ is the compact quantum group defined as the dual object to the following Hopf $C^*$-algebra:

\begin{definition}
The algebra $C(O_n^{-1})$ is the universal $C^*$-algebra generated by the entries of a $n\times n$ unitary matrix of self-adjoint elements $u_{ij}$ satisfying
\begin{enumerate}
\item distinct $u_{ij}$ anticommute on each row and column,
\item all the other pairs of $u_{ij}$ entries commute,    
\end{enumerate}
with coalgebra structure given by $\Delta(u_{ij})=\sum_ku_{ik}\otimes u_{kj}$, $\varepsilon(u_{ij})=\delta_{ij}$, $S(u_{ij})=u_{ji}$.
\end{definition}

We refer to \cite{bbc} for full details regarding this construction. 

In order to investigate stationarity questions for $O_n^{-1}$, our idea will be that of using techniques from \cite{bch}, based on the study of various central subalgebras of $C(O_n^{-1})$. 

Let us begin our study with the following standard result:

\begin{proposition}
The algebra of central monomials $A\subset C(O_n^{-1})$, i.e. the span of the products of coordinates $u_{ij}$ which are central, is given by the following formula, 
$$A=\overline{span}\left\{\prod_{ij}u_{ij}^{e_{ij}}\Big|{\rm the\ exponent\ matrix\ }e{\rm\ is\ bistochastic\ mod\ } 2\right\}$$
with the convention that the products can be arbitrarily expanded into usual monomials.
\end{proposition}

\begin{proof}
Consider an arbitrary monomial $z=u_{i_1j_1}\ldots u_{i_pj_p}$. According to the commutation relations inside $C(O_n^{-1})$, we have $zu_{ij}=(-1)^Eu_{ij}z$, with the exponent being:
$$E=\#\left\{s\Big|i_s=i,j_s\neq j\right\}+\#\left\{s\Big|i_s\neq i,j_s=j\right\}$$

Now if we formally compact our monomial as $z'=\prod_{ij}u_{ij}^{e_{ij}}$, with the exponents $e_{ij}\in\mathbb N$ being obtained by gathering the indices $(i_s,j_s)$, the above exponent becomes:
$$E=\sum_{k\neq j}e_{ik}+\sum_{k\neq i}e_{kj}$$

Since we are only interested in the value of $E$ mod 2, we can add if we want the missing terms, corresponding to the quantity $2e_{ij}=0$ mod 2, and we conclude that we have the following formula, with $R_i,C_j$ being the row and column sums of $e\in M_n(\mathbb N)$:
$$E=R_i+C_j\ {\rm mod}\ 2$$

We conclude that $z$ is central precisely when $R_i+C_j=0$ mod 2 for any $i,j$, and so when the exponent matrix $e\in M_n(\mathbb N)$ is bistochastic mod 2, as stated.
\end{proof}

In order to have more insight into the structure of $A$, consider the diagonal subgroup $L_n=\mathbb Z_2^n$ of $O_n^{-1}$. The left and right multiplication by $L_n$ induce coactions as follows:
$$\lambda:C(O_n^{-1})\to C(L_n)\otimes C(O_n^{-1})\quad,\quad \rho: C(O_n^{-1})\to C(O_n^{-1})\otimes C(L_n)$$

Our algebra $A$ then splits naturally as $A_0\oplus A_1$, where: 
$$A_\varepsilon=\overline{span}\left\{\prod_{ij}u_{ij}^{e_{ij}}\Big|e={\rm bistochastic\ mod\ } 2,\ {\rm with\ sums}\ \varepsilon\right\}$$

The point now is that these two components $A_\varepsilon$ can be recovered as follows, with the elements $g_\varepsilon\in L_n$ being given by $g_0=1$ and $g_1=(1,\ldots,1)$:
$$A_\varepsilon=\left\{x\in C(O_n^{-1})\Big|\lambda(x)=g_\varepsilon\otimes x,\ \rho(x)=x\otimes g_\varepsilon\right\}$$

Now recall from \cite{bbc} that $C(O_n^{-1})$ is the twist of $C(O_n)$ by a cocycle $\sigma:C(O_n)^{\otimes 2}\to \mathbb{C}$ which factors through the following surjection:
$$C(O_n)^{\otimes 2} \to C(L_n)^{\otimes 2}$$

In other words, $C(O_n^{-1})$ can be identified with $C(O_n)$ as a vector space, and in fact as a coalgebra, with new multiplication given, in Sweedler notation, by:
$$x\cdot y=\sigma^{-1}(x_1,y_1)~x_2y_2~\sigma(x_3,y_3)$$

Our remark here is that the subalgebra $A\subset C(O_n)$ ``survives'' the deformation, i.e. its cocycle-twisted counterpart in $C(O_n^{-1})$ retains the old multiplication: 

\begin{proposition}
The canonical vector space identification of $A\subset C(O_n)$ with its cocycle-twisted counterpart is an isomorphism of algebras. 
\end{proposition}

\begin{proof}
This follows from the characterization of $A$ by means of the left and right coactions $\lambda,\rho$ discussed above. Indeed, from $\sigma(1,-) = \sigma(-,1)=1$ we obtain that the twisted multiplication of $x\in A_0$ and $y\in A_\varepsilon$ is given by:
$$x\cdot y=\sigma^{-1}(1,g_\varepsilon)\,xy\,\sigma(1,g_\varepsilon) = xy$$
  
An analogous argument settles the case $y\in A_0$. Finally, if both $x$ and $y$ are elements of $A_1$, then, since the two factors involving $\sigma$ cancel out, we obtain:
$$x\cdot y=\sigma^{-1}(g_\varepsilon,g_\varepsilon)\,xy\,\sigma(g_\varepsilon,g_\varepsilon)=xy$$

Thus, the twisted and untwisted multiplicative structures coincide on $A$, as claimed.
\end{proof}

By performing now the spectrum computation inside $C(O_n)$, we obtain:

\begin{theorem}
The spectrum of the following subalgebra of $A$,
$$A_0=\overline{span}\left\{\prod_{ij}u_{ij}^{e_{ij}}\Big|e={\rm bistochastic\ mod\ } 2,\ {\rm with\ sums}\ 0\right\}$$
is the space $X_0=\mathbb Z_2^n\backslash O_n/\mathbb Z_2^n$. A similar result holds for $X=Spec(A)$.
\end{theorem}

\begin{proof}
As a first remark, in the context of Proposition 6.2, the possible row and column sums of $e$ being $0,1$, we have a decomposition $A=A_0\oplus A_1$, which is a $\mathbb Z_2$-grading. 

Regarding now the computation of the spectrum, recall first that we have:
$$C(PO_n)=\overline{span}\left\{\prod_{ij}u_{ij}^{e_{ij}}\Big|e={\rm with\ total\ sum\ zero\ mod\ } 2\right\}$$

The spectra $X_0,X$ of $A_0,A$ appear in a similar way, with $O_n$ being divided by a number of copies of $\mathbb Z_2$. To be more precise, in what regards $X_0$, the $n+n$ conditions which are imposed on $e$ correspond to the $n+n$ actions of $\mathbb Z_2$ by switching the signs on the $n+n$ rows and columns. As for $X$, the situation here is similar.
\end{proof}

\section{Stationarity results}

The goal of the present section is to construct a stationary model for $C(O_2^{-1})$, using the material of \cite{bch}. To be more precise, our result regarding $O_2^{-1}$ is as follows:

\begin{theorem}
We have matrix model of type
$$C(O_2^{-1})\to M_4(C(\mathbb T))$$
which is stationary with respect to the uniform measure on $\mathbb T$.
\end{theorem}

\begin{proof}
We use the well-known cocycle twisting picture of $G=O_2^{-1}$. Let $L_2=\mathbb{Z}_2^2$ be the diagonal subgroup of $O_2$. Since $C(G)$ is obtained from $C(O_2)$ by twisting by a cocycle $\sigma:C(O_2^2)^{\otimes 2}\to \mathbb{C}$ that factors through $C(G)^{\otimes 2}\to C(L_2)^{\otimes 2}\to \mathbb{C}$, the group $L_2$ survives the deformation and appears as a quantum subgroup of $G$ as well. In other words, we obtain a surjection of Hopf $*$-algebras, as follows:
$$p:C(G)\to C(L_2)$$

Consider now the algebra $A$ consisting of the elements $x\in C(G)$ satisfying:
$$x_1\otimes p(x_2) = x\otimes 1\in C(G)\otimes C(L_2)$$

Attached to the surjection $p$ we then have a homogeneous space, as follows:
$$\iota:A=C(G/L_2)\to C(G)$$

More concretely now, denoting the standard matrix generators of $C(G)$ by $u_{ij}$ with $1\le i,j\le 2$, our algebra $A$ is the subalgebra generated by the following monomials:
$$\prod_{i,j}u_{ij}^{e_{ij}},\ e_{1j}+e_{2j}\text{ is even for }j=1,2$$

We know from Theorem 6.4 that this subalgebra $A\subset C(G)$ is commutative, and equals the function algebra of the space $X=O_2/L_2$. Despite the fact that $A\subset C(G)$ is not a Hopf subalgebra but rather only a left comodule subalgebra, the proof of \cite[Theorem 2.5]{bch} goes through virtually verbatim to prove that the regular action of $C(G)$ on itself by left multiplication embeds it into the bundle $End_{A}C(G)$ of $4\times 4$ matrix algebras over the homogeneous space $X=\mathrm{Spec}(A)$. Thus, we have an embedding as follows:
$$\pi:C(G)\to End_{A}C(G)$$

Moreover, once again by reasoning as in \cite{bch}, we deduce that this embedding is stationary with respect to the Haar measure on $X$ and to the normalized trace:
$$tr:End_{C(L)}C(G)\to C(L)$$

Note furthermore that $X$ is isomorphic to the circle group obtained by quotienting out the maximal torus of $O_2$ by its order-two subgroup. This implies that the four-dimensional bundle over $X$ associated to the $A$-module $C(G)$ is trivial, and hence we have:
$$End_A(C(G))\simeq M_4(A)$$

All in all, we obtain a stationary matrix model, as in the statement.
\end{proof}

As explained in \cite{bbc} we have an embedding $O_n^{-1}\subset S_{2^n}^+$ for any $n\in\mathbb N$, obtained by viewing $O_n^{-1}$ as the quantum symmetry group of the $n$-hypercube. In particular, we can talk about the universal flat model for $O_n^{-1}$, and deciding whether this model is stationary or not is therefore an interesting question, which makes sense at any $n\in\mathbb N$.

In connection with the above considerations, let us note that we have an embedding of $L_n=\mathbb{Z}_n^2$ into $O_n^{-1}$, for any $n\in\mathbb N$. However, the resulting quantum homogeneous space $O_n^{-1}/L_n$ is no longer classical at $n\ge 3$, and hence the above method does not apply directly to produce matrix models for higher-dimensional twisted orthogonal groups.

\end{document}